\documentclass[13pt]{amsart}

\usepackage{amsmath,amssymb,amsxtra,latexsym,amsthm}
\usepackage{amsfonts}
\usepackage{multirow}
\usepackage[british]{babel}
\usepackage{amssymb,amsmath}
\usepackage{amsfonts}
\usepackage{amsmath}
\usepackage{enumerate}
\usepackage{mathrsfs}
\usepackage[pdftex]{graphicx,color}
\usepackage{colortbl}
\usepackage[table]{xcolor}
\usepackage[numbers]{natbib}
\usepackage{float}
\usepackage[section]{placeins}
\usepackage[normalem]{ulem}

\makeatletter
\@namedef{subjclassname@2010}{%
\textup{2010} Mathematics Subject Classification}
\makeatother

\newtheorem{theorem}{Theorem}[section]

\newtheorem{lemma}[theorem]{Lemma}
\newtheorem{definition}[theorem]{Definition}

\theoremstyle{definition}

\newtheorem*{theorem*}{Theorem}
\newtheorem*{acknowledgements*}{Acknowledgements}

\newcommand{\F}{\mathbb{F}}

\numberwithin{equation}{section}

\begin{document}

\title[Plane sections of Fermat surfaces over finite fields]{Plane sections of Fermat surfaces over finite fields}

\author[H. Borges]{Herivelto Borges}
\address{Instituto de Ci\^encias Matem\'aticas e de Computa\c c\~ao\\ Universidade de S\~ao Paulo\\
Avenida Trabalhador S\~ao-carlense, 400, CEP 13566-590, S\~ao Carlos SP, Brazil}
\email{hborges@icmc.usp.br}

\author[G. Cook]{Gary Cook}
\address{Instituto de Ci\^encias Matem\'aticas e de Computa\c c\~ao\\ Universidade de S\~ao Paulo\\
Avenida Trabalhador S\~ao-carlense, 400, CEP 13566-590, S\~ao Carlos SP, Brazil}
\email{garycook82@msn.com}

\author[M. Coutinho]{Mariana Coutinho}
\address{Instituto de Ci\^encias Matem\'aticas e de Computa\c c\~ao\\ Universidade de S\~ao Paulo\\
Avenida Trabalhador S\~ao-carlense, 400, CEP 13566-590, S\~ao Carlos SP, Brazil}
\email{mariananery@usp.br}

\date{}



\maketitle

\begin{abstract}
In this paper, we characterize all curves over $\F_q$ arising from a plane section
$$
\mathcal{P} : X_3-e_0X_0-e_1X_1-e_2X_2 = 0
$$
of the Fermat surface
$$
\mathcal{S} : X_0^d + X_1^d + X_2^d +X_3^d = 0,
$$
where $q = p^{h} = 2d+1$ is a prime power, $p >3$, and $e_0, e_1, e_2 \in \mathbb{F}_q$. In particular, we will prove that any nonlinear component $\mathcal{G} \subseteq \mathcal{P} \cap \mathcal{S} $ is a smooth classical curve of degree $n\leqslant d$ attaining the  St\"ohr-Voloch bound
$$
\# \mathcal{G}(\mathbb{F}_q) \leqslant \frac{1}{2} n(n+q-1) - \frac{1}{2} i(n-2),
$$
with $i \in \{0,1,2,3,n,3n\}$.
\end{abstract}

\section{Introduction}

Let $\mathcal{F}$ be the curve obtained by slicing the Fermat surface  
$$
\mathcal{S} : X_0^d + X_1^d + X_2^d + X_3^d = 0
$$
with the plane  
$$
\mathcal{P} : X_3-e_0X_0-e_1X_1-e_2X_2 = 0,
$$
where $d$ is a positive integer, $e_0, e_1, e_2 \in \mathbb{F}_q$, and $\mathbb{F}_q$ is the finite field with $q=p^h$ elements, with $p$ a prime number. In other words, let
\begin{eqnarray}
\label{curveF}
\mathcal{F}:    X_0^d + X_1^d + X_2^d +(e_0 X_0 + e_1 X_1 + e_2 X_2)^d = 0.
\end{eqnarray}

Characterizing this general curve $\mathcal{F}$ in terms of its rational points and its irreducible and nonsingular components presents many challenges.
For instance, the particular case $p=2$ and $e_0=e_1=e_2=1$ has been extensively investigated over the past decades (see  \cite{HMc}, \cite{JMcW}, \cite{JW}, \cite{LW}). In this context, the following result was essential in Hernando and McGuire's proof of  an important conjecture regarding exceptional numbers \cite{HMc}.
\begin{theorem*}[Hernando-McGuire]
The polynomial $$\frac{X_0^d + X_1^d + X_2^d +(X_0 + X_1 + X_2)^d}{(X_0+X_1)(X_0+X_2)(X_1+X_2)}$$ 
 has an absolutely irreducible factor defined over $\mathbb{F}_2$ for all $d$ not of the form $d=2^{i}+1$ or $d=2^{2i}-2^{i}+1.$
\end{theorem*}

In this paper, we consider the problem of studying the curve given in \eqref{curveF} from another point of view. Based on techniques developed by Carlin and Voloch \cite{CV}, we characterize the curve
\begin{eqnarray}
\label{curveC}
\hspace{1.6cm} \mathcal{C}: C(X_0,X_1,X_2) =  X_0^d + X_1^d + X_2^d +(e_0 X_0 + e_1 X_1 + e_2 X_2)^d = 0,
\end{eqnarray}
where $q = p^{h} = 2d+1$ is a prime power, $p > 3$, and $e_0, e_1$ and $e_2$ are arbitrary elements in $\mathbb{F}_q$.
For such a curve, we give a complete description of the irreducible and nonsingular components and provide their number of $\F_q$-rational points. 
Consequently, we construct a family of curves attaining the St\"ohr-Voloch bound and prove the following theorem, which is the main result of this paper.

\begin{theorem} \label{Thmain}
If $\mathcal{C}$ is not the union of $d$ lines, then the following statements hold.
\begin{enumerate}
\item[{\rm{(i)}}] The curve $\mathcal{C}$ is the union of $N\in\{0,1,2,3\}$  $\mathbb{F}_q$-lines and a nonsingular classical curve $\mathcal{G}$ of degree $n$;
\item[{\rm{(ii)}}] The possibilities for $\# \mathcal{G} (\mathbb{F}_q)$ are
$$
\frac{1}{2} n(n+q-1) - \frac{1}{2} i(n-2),
$$
with $i \in \{ 0,1,2,3,n,3n \}$. In particular, curve $\mathcal{G}$ meets the St\"ohr-Voloch bound in \emph{(\ref{chapa})}.
\end{enumerate}
\end{theorem}

It is worth mentioning that few examples of curves attaining the St\"ohr-Voloch bound are known.
Such explicit constructions are of interest in areas such as Finite Geometry and Coding Theory. 

The present work is organized as follows. In Section \ref{SecPtsLns}, the $\mathbb{F}_q$-points and linear components of $\mathcal{C}$ are detailed. In particular, it is shown that if $\mathcal{C}$ is not the union of $d$ lines, then 
$$
C(X_0,X_1,X_2) = (e_0X_0 + e_1X_1)^{i_0}(e_0X_0+ e_2X_2)^{i_1}(e_1X_1+e_2X_2)^{i_2} G(X_0,X_1,X_2),
$$
where $i_0,i_1,i_2 \in \{ 0,1 \}$ and $G(X_0,X_1,X_2)$ has no linear factors. In Section \ref{SecPre}, it is shown that the curve $\mathcal{G} : G(X_0,X_1,X_2) = 0$ is $\mathbb{F}_q$-disjoint from any linear component of $\mathcal{C}$. These facts are used  to prove Theorem \ref{Thmain}, which relies on important results obtained in \cite{CV}, \cite{RV} and \cite{SV}. Full details are given in Subsection 4.1. For general background on curves over finite fields, see \cite{H} and \cite{HKT}.\\

\begin{center}
\bf{Notation}
\end{center}

The following notation will be used throughout this text.

\begin{itemize}
\item The number of points of $\mathbb{P}^2(\mathbb{F}_q)$ on a curve $\mathcal{F}$ defined over $\mathbb{F}_q$ is denoted by $\#\mathcal{F} (\mathbb{F}_q)$.

\item The quadratic character on $\mathbb{F}_q$ is denoted by $\eta$; that is,
	$$
	\eta(u) = 
	\begin{cases}
	1,  &\text{if } u \text{ is a nonzero square} \\
	0,  &\text{if } u = 0 \\
	-1, &\text{if } u \text{ is a non-square,}
	\end{cases}
	$$
	for all $u \in \mathbb{F}_q$. In particular, as $d=\frac{q-1}{2}$, we have $\eta(u)=u^d$, for all $u \in \mathbb{F}_q$.

\item By the notation $\{\eta(e_0),\eta(e_1),\eta(e_2)\}=\{0,1,1\}$ it is meant that two of the values $\eta(e_0),\eta(e_1),\eta(e_2)$ are equal to 1 and one is equal to 0; similarly for other cases.

\item The points $(1:0:0)$, $(0:1:0)$ and $(0:0:1)$ in $\mathbb{P}^2(\mathbb{F}_q)$ are denoted by $P_0$, $P_1$ and $P_2$, respectively. 
The points $(-e_1:e_0:0)$, $(-e_2:0:e_0)$ and $(0:-e_2:e_1)$ are denoted by $P_{01}=P_{10}$, $P_{02}=P_{20}$ and $P_{12}=P_{21}$, respectively.

\item The sets $$\{(0:x_1:1)\in \mathbb{P}^2(\mathbb{F}_q)\ |  \ \eta(x_1)=-1\},$$ $$\{(x_0:0:1)\in \mathbb{P}^2(\mathbb{F}_q)\ | \  \eta(x_0)=-1\}$$ and $$\{(x_0:1:0)\in \mathbb{P}^2(\mathbb{F}_q)\ |  \ \eta(x_0)=-1\}$$ are denoted by $A_0$, $A_1$ and $A_2$, respectively.
\end{itemize}

\vspace{-0.5cm}

\section{Points and linear components of curve $\mathcal{C}$} \label{SecPtsLns}

In this section, the $\mathbb{F}_q$-points and linear components of  curve $\mathcal{C}$ given in \eqref{curveC} are investigated.

\subsection{Points with zero coordinates} \label{App-pts-zero}

The following lemma follows from the definitions.

\begin{lemma} \label{lem-cur-pts}
If $e_0=e_1=e_2=0$, then a point $(x_0:x_1:x_2) \in \mathbb{P}^2(\mathbb{F}_q)$ is on curve $\mathcal{C}$ if and only if $\{\eta(x_0),\eta(x_1),\eta(x_2)\}=\{-1,0,1\}$. 
In particular, $\#\mathcal{C}(\mathbb{F}_q)=3d$.
\end{lemma}

\begin{lemma} \label{pts}
For a point $P=(x_0:x_1:x_2)\in \mathbb{P}^2(\mathbb{F}_q)$ with $x_0x_1x_2=0$, the following occurs.
\begin{enumerate}
\item[\emph{(i)}] For $i=0,1,2$,  
$$P_i \in \mathcal{C}(\mathbb{F}_q)  \text{ if and only if  }\eta(e_i)=-1;$$
\item[\emph{(ii)}] Let $i,j \in \{0,1,2\}$ be distinct elements such that $x_ix_j \neq 0$. Then $P\in\mathcal{C}(\mathbb{F}_q)$ if and only if 
$$
e_i=e_j=0 {\text{ and }} \eta (x_ix_j)=-1
$$
or
$$
\eta(-e_ie_j)=-1 \text{ and }P=P_{ij}.
$$
\end{enumerate}
\end{lemma}

\begin{proof}
Let $P=(x_0:x_1:x_2)$ be a point of $\mathbb{P}^2(\mathbb{F}_q)$ with $x_0x_1x_2=0$. 
Statement (i) follows directly from the definition of $\mathcal{C}$.
Suppose without loss of generality that $x_0 x_1 \neq 0$ and $x_2 = 0$.
Then the following assertions are equivalent:
\begin{itemize}
\item $P \in \mathcal{C}(\mathbb{F}_q)$;
\item $x_0^d+x_1^d+(e_0x_0+e_1x_1)^d=0$;
\item $e_0x_0+e_1x_1=0 \text{ and } \eta(x_0x_1)=-1$;
\item $e_0=e_1=0$ and $\eta(x_0x_1)=-1$ or \\ $\eta(-e_0e_1)=-1$ and $P=P_{01}$,
\end{itemize}
which completes the proof.
\end{proof}

Based on these results, the number of $\mathbb{F}_q$-points with zero coordinates is summarized in Table \ref{Table1}, for all possible quadratic characters of $e_0$, $e_1$ and $e_2$.

\begin{table}[H]
\centering
\begin{tabular}{| c | c | c | c |}
\hline
$\{ \eta(e_0), \eta(e_1), \eta(e_2) \}$     &    $i=2$   &    $i=1$ and $d$ odd    &    $i=1$ and $d$ even\\
\hline
\{ 1,1,1 \}		&    0    &    3		&    0\\
\{ -1,1,1 \}		&    1    &    1		&    2\\ 
\{ -1,-1,1 \}		&    2    &    1		&    2\\
\{ -1,-1,-1 \}		&    3    &    3		&    0\\
\{ 0,1,1 \}		&    0    &    1		&    0\\
\{ -1,0,1 \}		&    1    &    0		&    1\\
\{ -1,-1,0 \}		&    2    &    1		&    0\\
\{ 0,0,1 \}		&    0    &    $d$	&    $d$\\
\{ -1,0,0\}		&    1    &    $d$	&    $d$\\
\{ 0,0,0 \}		&    0    &    $3d$	&  $3d$\\
\hline
\end{tabular}
\caption{\small{The number of $\mathbb{F}_q$-points with $i\in \{1,2\}$ zero coordinates.}}
\label{Table1}
\label{Tab-pts}
\end{table}

\vspace{-0.9cm}

\subsection{Points without zero coordinates} \label{App-pts-no}

A point $P = (1:x_1:x_2) \in \mathbb{P}^2(\mathbb{F}_q)$, with $x_1x_2 \neq 0$, is on  $\mathcal{C}(\F_q)$ if and only if one of the following cases occurs:
\begin{eqnarray*}
(1)&& \eta(x_1) = 1, \eta(x_2) = -1 \text{ and } \eta(e_0 + e_1x_1 + e_2x_2) = -1;\\
(2)&& \eta(x_1)= -1, \eta(x_2) = 1 \text{ and } \eta(e_0 + e_1x_1 + e_2x_2) = -1;\\
(3)&& \eta(x_1) = -1, \eta(x_2) = -1 \text{ and } \eta(e_0 + e_1x_1 + e_2x_2) = 1.
\end{eqnarray*}

If $N_{(i)}$ is the number of points satisfying case $(i)$, for $i \in \{1,2,3\}$, then $N_{(1)}$, $N_{(2)}$ and $N_{(3)}$ are the cardinalities of the sets
$$
\{(y_1^2,\lambda y_2^2,\lambda y_3^2) \in {\mathbb{F}_q^{\ast}}^3\ | \ e_0+ e_1y_1^2+e_2\lambda y_2^2=\lambda y_3^2\},
$$
$$
\{(\lambda y_1^2, y_2^2,\lambda y_3^2) \in {\mathbb{F}_q^{\ast}}^3 \ | \  e_0+ e_1\lambda y_1^2+e_2 y_2^2=\lambda y_3^2\},
$$
and
$$
\{(\lambda y_1^2,\lambda y_2^2, y_3^2) \in {\mathbb{F}_q^{\ast}}^3 \ | \  e_0+ e_1\lambda y_1^2+e_2\lambda y_2^2=y_3^2\},
$$
respectively, where $\lambda$ is a fixed element of $\mathbb{F}_q$ satisfying $\eta(\lambda)=-1$, and $\mathbb{F}_q^{\ast}=\mathbb{F}_q\setminus \{0\}$. 

Let
$$
N\bigg( \sum_{j=1}^s {b_j Y_j^2} = \beta \bigg)
$$ 
represent the number of solutions $(y_1,\ldots,y_s) \in \mathbb{F}^s_q$ of the equation $$\displaystyle\sum_{j=1}^s b_j Y_j^2 = \beta$$ defined over $\mathbb{F}_q$, and let
$$
a_1=e_1, a_2=\lambda e_2, a_3=-\lambda,\alpha=-e_0, \text{ if } i=1,
$$
$$
a_1=\lambda e_1, a_2=e_2, a_3=-\lambda,\alpha=-e_0, \text{ if } i=2, 
$$
and 
$$
a_1=\lambda e_1, a_2=\lambda e_2, a_3=-1,\alpha=-e_0, \text{ if } i=3.
$$
Then the number $N_{(i)}$, with $i \in\{1,2,3\}$, is determined by the following expressions, for all possible cases of $e_0, e_1$ and $e_2$.
\begin{enumerate}[\rm(i)]
\item {{$\text{Case }e_0e_1e_2\neq 0$}}
\begin{eqnarray*}
8N_{(i)}&=&N(a_1Y_1^2+a_2Y_2^2+a_3Y_3^2=\alpha)-\sum_{j=1}^{3}N\bigg(\sum_{k=1, k\neq j}^{3}a_kY_k^2=\alpha\bigg)\\
&+&\sum_{j=1}^{3}N(a_jY_j^2=\alpha).
\end{eqnarray*}
\item {{$\text{Case }e_0e_j\neq 0, e_k=0,\text{ with } j,k\in\{1,2\}$}}
$$
8N_{(i)}=(q-1)\bigg[N(a_jY_1^2+a_3Y_3^2=\alpha)-\sum_{t=1,t\neq k}^{3}N(a_tY_t^2=\alpha)\bigg].
$$
\item {{$\text{Case }e_1e_2\neq 0, e_0=0$}}
\begin{eqnarray*}
8N_{(i)}&=&N(a_1Y_1^2+a_2Y_2^2+a_3Y_3^2=\alpha)-\sum_{j=1}^{3}N\bigg(\sum_{k=1,k\neq j}^{3}a_kY_k^2=\alpha\bigg)\\
&+&\sum_{j=1}^{3}N(a_jY_j^2=\alpha)-1.
\end{eqnarray*}
\item {{$\text{Case }e_0\neq 0, e_1=e_2=0$}}
$$
8N_{(i)}=(q-1)^2N(a_3Y_3^2=\alpha).
$$
\item {{$\text{Case }e_j\neq 0, e_0=e_k=0,\text{ with } j,k\in\{1,2\}$}}
$$
8N_{(i)}=(q-1)\bigg[N(a_jY_j^2+a_3Y_3^2=\alpha)-\sum_{t=1,t\neq k}^{3}N(a_tY_t^2=\alpha)+1\bigg].
$$
\end{enumerate}

Table \ref{tab:N(1)+N(2)+N(3)} presents the values of $N_{(1)}+N_{(2)}+N_{(3)}$, which are calculated using the following lemma given by Propositions 1 and 2 in \cite[Chapter 6]{J}.

\begin{lemma}
\label{prop1&2}
Let $p>2$ and $s$ a positive integer. The number of solutions in $\mathbb{F}_q^s$ of the equation $b_1Y_1^2+\cdots+b_sY_s^2=\beta$, with $b_j\in\mathbb{F}_q^{\ast}$ for all $j=1,\ldots,s$, is given by:
$$
\begin{cases}
q^{s-1}+\eta((-1)^{s/2}b_1\cdots b_s)(q^{s/2}-q^{(s-2)/2}), &\text{if } \beta=0 \text{ and } s\equiv 0 \mod 2, \\
q^{s-1}-\eta((-1)^{s/2}b_1\cdots b_s)q^{(s-2)/2}, &\text{if } \beta \neq 0 \text{ and } s\equiv 0 \mod 2,\\
q^{s-1}, &\text{if } \beta =0 \text{ and } s\equiv 1 \mod 2,\\
q^{s-1}+\eta((-1)^{(s-1)/2}b_1\cdots b_s \beta)q^{(s-1)/2},  &\text{if } \beta\neq 0 \text{ and } s\equiv 1 \mod 2.
\end{cases}
$$
\end{lemma}

\begin{table}[H]
\centering
\begin{tabular}{| c | c | c |}
\hline
\multicolumn{3}{|c|}{$e_0e_1e_2 \neq 0$} \\
\hline
$\{ \eta(e_0), \eta(e_1), \eta(e_2) \}$   &    $d$ odd     &   $d$ even\\
\hline
\{1,1,1\}	&    $\frac{3(q-1)(q-3)}{8}$	&    $\frac{3(q-1)^2}{8}$ \\   
\{-1,1,1\}	&    $\frac{3q^2 - 6q + 7}{8}$	&    $\frac{3(q-1)(q-3)}{8}$ \\    
\{-1,-1,1\}	&    $\frac{3(q-1)(q-3)}{8}$	&    $\frac{3q^2 - 6q + 11}{8}$ \\    
\{-1,-1,-1\}	&    $\frac{3(q^2 - 2q + 5)}{8}$	&    $\frac{3(q-1)(q-3)}{8}$ \\
\hline
\multicolumn{3}{|c|}{Exactly one of the elements  $e_0,e_1,e_2$ is zero} \\
\hline
$\{ \eta(e_0), \eta(e_1), \eta(e_2) \}$   &    $d$ odd     &   $d$ even\\
\hline
\{0,1,1\}	&    $\frac{(q-1)(3q-5)}{8}$	&    $\frac{3(q-1)^2}{8}$ \\
\{-1,0,1\}	&    $\frac{(q-1)(3q-5)}{8}$	&    $\frac{(q-1)(3q-7)}{8}$ \\
\{-1,-1,0\}	&    $\frac{3(q-1)(q-3)}{8}$	&   $\frac{(q-1)(3q-7)}{8}$ \\   
\hline
\multicolumn{3}{|c|}{Exactly two of the elements $e_0,e_1,e_2$ are zero} \\
\hline
$\{ \eta(e_0), \eta(e_1), \eta(e_2) \}$   &    $d$ odd     &   $d$ even\\
\hline
\{0,0,1\}	&    $\frac{(q-1)^2}{4}$		&    $\frac{(q-1)^2}{4}$ \\    
\{-1,0,0\}	&    $\frac{(q-1)^2}{2}$		&    $\frac{(q-1)^2}{2}$ \\    
\hline
\end{tabular}
\caption{\small{The number $N_{(1)}+N_{(2)}+N_{(3)}$.}}
\label{tab:N(1)+N(2)+N(3)}
\end{table}

\vspace{-0.9cm}

\subsection{Linear components} \label{App-Lines}

\begin{lemma}\label{lem-comps}
The linear components of $\mathcal{C}$ and the circumstances in which they arise are as follows:
\begin{enumerate}
\item[{\rm{(i)}}]  The curve  $\mathcal{C}$ is a union of $d$ lines  if and only if one of $\eta (e_0),\eta (e_1),\eta (e_2)$ is $-1$ and the other two are zero;
\item[{\rm{(ii)}}] If $\mathcal{C}$ is not a union of $d$ lines, then a line $\ell$ is a component of $\mathcal{C}$  if and only if $\ell$ is given by $e_iX_i+e_jX_j=0$ and $\eta (-e_ie_j) = \eta (e_k)= -1$, with $\{i,j,k\}=\{0,1,2\}$.
\end{enumerate}
\end{lemma}

\begin{proof}
Let $\ell_0=P_0P_1$, $\ell_1=P_0P_2$, $\ell_2=P_1P_2$, and let $\ell:x_0X_0+x_1X_1+x_2X_2=0$ be a linear component of $\mathcal{C}$. Since none of the lines $\ell_i$ is a linear component of $\mathcal{C}$, we have that $\ell$ must intersect each of the three lines $\ell_0$, $\ell_1$ and $\ell_2$ at the points $(x_1:-x_0:0)$, $(x_2:0:-x_0)$ and $(0:x_2:-x_1)$ on $\mathcal{C}$, respectively. The proof follows directly from Lemma \ref{pts} considering every possibility of $\{ \eta(e_0), \eta(e_1), \eta(e_2) \}$.
\end{proof}

The table below summarizes the linear components of $\mathcal{C}$ when it is not the union of $d$ lines.
\begin{table}[H]
\centering
\scalebox{0.825}{
\begin{tabular}{| c | c | c |} 
\hline
$(\eta(e_0),\eta(e_1),\eta(e_2))$     &     $d$ odd     &     $d$ even\\
\hline
$(1,1,1)$	&     -     									&     -\\
$(1,1,-1)$	&     $e_0X_0+e_1X_1$							&     -\\    
$(1,-1,1)$	&     $e_0X_0+e_2X_2$							&     -\\
$(-1,1,1)$	&     $e_1X_1+e_2X_2$							&     -\\  
$(1,-1,-1)$	&     -     									&     $e_0X_0+e_1X_1, e_0X_0+e_2X_2$\\
$(-1,1,-1)$	&     -     									&     $e_0X_0+e_1X_1, e_1X_1+e_2X_2$\\
$(-1,-1,1)$	&     -     									&     $e_0X_0+e_2X_2, e_1X_1+e_2X_2$\\    
$(-1,-1,-1)$	&     $e_0X_0+e_1X_1, e_0X_0+e_2X_2, e_1X_1+e_2X_2$	&     -\\
\hline
\end{tabular}
}
\caption{\small{Linear components of curve $\mathcal{C}$ for $e_0e_1e_2\neq 0$.}}
\label{tab:linesabcdiff0}
\end{table}      

\vspace{-0.75cm}

\begin{lemma} \label{lem-mult}
If $e_0,e_1$ and $e_2$ are all nonzero, then the linear components of $\mathcal{C}$ have multiplicity at most 1. That is, none of  $(e_0X_0 +e_1X_1)^2 = 0$, $ (e_0X_0 + e_2X_2)^2 = 0$ or $ (e_1X_1 + e_2X_2)^2 = 0$ is a component of $\mathcal{C}$.
\end{lemma}

\begin{proof}
Without loss of generality, assume that $\ell_0: e_0X_0 + e_1X_1 = 0$ is a component of $\mathcal{C}$, and then $\eta(e_2)=-1$ by Lemma \ref{lem-comps}. 
Since
\begin{eqnarray*}
C(X_0,X_1,X_2)&=&X_0^d + X_1^d + X_2^d + (e_0X_0 + e_1X_1)^d + (e_2X_2)^d + \\
& &+\sum_{i=1}^{d-1}\binom{d}{i} (e_0X_0 + e_1X_1)^i (e_2X_2)^{d-i},
\end{eqnarray*}
it follows that $(e_0X_0 + e_1X_1) \mid (X_0^d + X_1^d)$. If $(e_0X_0 + e_1X_1)^2 = 0$ is a component of $\mathcal{C}$ and
$$
(e_0X_0 + e_1X_1)^2 \mid (X_0^d + X_1^d + d(e_0X_0+e_1X_1)(e_2X_2)^{d-1}),
$$
then $e_2=0$, which is a contradiction.
\end{proof}

\section{Preliminary result} \label{SecPre}

\begin{theorem} \label{Thcom1}
Let $\mathcal{C}$ be the curve given in \emph{\eqref{curveC}}. The union of the linear components of $\mathcal{C}$ is $\F_q$-disjoint from its remaining components.
\end{theorem}
\begin{proof}
Let $N \in \{0,1,2,3,d\}$ be the number of linear components of $\mathcal{C}$. If $N \in \{0,d\}$, then the proof is complete. Now each of the remaining cases is considered separately.

\begin{description}
\item[\it{Case $N=1$}] \hfill \\
Without loss of generality, let $\ell_0 : e_0X_0 + e_1X_1 = 0$ be the linear component of $\mathcal{C}$. Then, $d$ is odd, $\eta(e_0)=\eta(e_1)=1$, and $\eta(e_2) =\eta(-e_0e_1)= -1$ (see Table \ref{tab:linesabcdiff0}).

Assume that the  point $P = (x_0:x_1:x_2) \in \mathbb{P}^2(\mathbb{F}_q)$ lies on  $\ell_0$ and an additional component of $\mathcal{C}$. Then the polynomial 
$$
g(X_1) = x_0^d + X_1^d + x_2^d + (e_0x_0 + e_1X_1 + e_2x_2)^d
$$
vanishes at $x_1$ with multiplicity at least two. Therefore, 
$$
\frac{\text{d}g}{\text{d}X_1}(x_1)=dx_1^{d-1} + de_1(e_2x_2)^{d-1} = 0, 
$$
that is, $x_1^{d-1} + e_1(e_2x_2)^{d-1} = 0$.

If $x_1= 0$, then $x_0=x_2=0$. Since this is impossible, it follows that $x_1\neq 0$ and 
$$
e_1\bigg(\frac{e_2x_2}{x_1}\bigg)^{d-1} = -1.
$$

This implies that $\eta(-1)=1$, which contradicts the fact that $d$ is odd.

\item[\it{Case $N=2$}] \hfill \\
Without loss of generality, let $\ell_0 : e_0X_0 + e_1X_1 = 0$  and  $\ell_1 : e_0X_0+e_2X_2 = 0$ be the linear components of $\mathcal{C}$, and assume that $P = (x_0:x_1:x_2) \in \mathbb{P}^2(\mathbb{F}_q)$ lies on $\ell_0$ and an additional component of $\mathcal{C}$, but not on $\ell_1$. Then $d$ is even, $\eta(e_0)=1$, $\eta(e_1) =\eta(e_2) =-1$ (see Table \ref{tab:linesabcdiff0}), and the polynomial 
$$
g(X_1) = x_0^d + X_1^d + x_2^d + (e_0x_0 + e_1X_1 + e_2x_2)^d
$$ 
vanishes at $x_1$ with multiplicity at least two. Therefore, 
$$
\frac{\text{d}g}{\text{d}X_1}(x_1)=dx_1^{d-1} + de_1(e_2x_2)^{d-1} = 0,
$$
that is, $x_1^{d-1} + e_1(e_2x_2)^{d-1} = 0$.

If $x_1 = 0$, then $x_0 = x_2 = 0$. Since this is impossible, it follows that $x_1\neq 0$ and 
$$
e_1\eta(x_1x_2)=\frac{e_2x_2}{x_1}.
$$

If $\eta(x_1x_2)=1$, then $e_1x_1=e_2x_2$ and $P$ lies on $\ell_1$, which is a contradiction. On the other hand, if $\eta(x_1x_2)=-1$, then $e_1=-\frac{e_2x_2}{x_1}$ and $\eta(e_1) = 1$, which is also a contradiction since $\eta(e_1) = -1$.

Assume that the point $P \in \ell_0 \cap \ell_1$ lies on an additional component of $\mathcal{C}$. Then
$$
h(X_0) = X_0^d + x_1^d + x_2^d + (e_0X_0 + e_1x_1 + e_2x_2)^d
$$
vanishes at $x$ with multiplicity at least three. Therefore,
$$
\frac{\text{d}^2 h}{\text{d}X_0^2}(x_0)=d(d-1)x_0^{d-2} + d(d-1)e_0^2(e_2x_2)^{d-2} = 0,
$$
that is, $x_0^{d-2} + e_0^2(e_2x_2)^{d-2} = 0$. However, since $e_0x_0 + e_2x_2 = 0$, $d$ is even and $\eta(e_0) = 1$, it follows that $2x_0^{d-2} = 0$, that is, $x_0 = x_1 = x_2 = 0$, which is impossible.

\item[\it{Case $N=3$}] \hfill \\
Let curve $\mathcal{C}$ have lines $\ell_0 : e_0X_0 + e_1X_1 = 0$, $\ell_1 : e_0X_0 + e_2X_2 = 0$, and $\ell_2 : e_1X_1 + e_2X_2 = 0$. Then  $\eta(e_0) =\eta(e_1) = \eta(e_2) = -1$ and $d$ is odd (see Table \ref{tab:linesabcdiff0}).

Without loss of generality, assume that $P = (x_0:x_1:x_2) \in \mathbb{P}^2(\mathbb{F}_q)$ lies on $\ell_0$ and an additional component of $\mathcal{C}$, but not on $\ell_1$ and $\ell_2$. Thus 
$$
g(X_1) = x_0^d + X_1^d + x_2^d + (e_0x_0 + e_1X_1 + e_2x_2)^d
$$
vanishes at $x_1$ with multiplicity at least two. Therefore, 
$$
\frac{\text{d}g}{\text{d}X_1}(x_1)=dx_1^{d-1} + de_1(e_2x_2)^{d-1} = 0, 
$$
that is, $x_1^{d-1} + e_1(e_2x_2)^{d-1} = 0$.

If $x_1 = 0$, then $x_0 = x_2 = 0$. Since this is impossible, it follows that $x_1\neq 0$ and 
$$
e_1\eta(x_1x_2)=\frac{e_2x_2}{x_1}.
$$

If $\eta(x_1x_2)=1$, then $e_1x_1=e_2x_2$ and $P \in \ell_1$, which is a contradiction. If $\eta(x_1x_2)=-1$, then $e_1x_1+e_2x_2=0$ and $P\in \ell_2$, which is also a contradiction.

Without loss of generality, assume that $P\in \ell_0\cap \ell_1$ lies on an additional component of $\mathcal{C}$, but $P\notin \ell_2$. Then
$$
h(X_0) = X_0^d + x_1^d + x_2^d + (e_0X_0 + e_1x_1 + e_2x_2)^d
$$
vanishes at $x$ with multiplicity at least three. Therefore,
$$
\frac{\text{d}^2 h}{\text{d}X_0^2}(x_0)=d(d-1)x_0^{d-2} + d(d-1)e_0^2(e_2x_2)^{d-2} = 0,
$$
that is, $x_0^{d-2} + e_0^2(e_2x_2)^{d-2} = 0$. However, since $e_0x_0 + e_2x_2 = 0$, $d$ is odd and $\eta(e_0) = -1$, it follows that $2x_0^{d-2} = 0$, that is, $x_0 = x_1 = x_2 = 0$, which is impossible.
\end{description}

Hence, for all possible cases, it has been shown that assuming $P \in \mathbb{P}^2(\mathbb{F}_q)$ to be a point on a linear component of $\mathcal{C}$ and on an additional nonlinear component of $\mathcal{C}$ leads to a contradiction.
\end{proof}

\section{Main result} \label{SecMain}

Before proving the central result of this work, Theorem \ref{Thmain}, in Subsection \ref{proofmainresult}, we need the results presented in the following subsection.

\subsection{Frobenius classicality and absolute irreducibility}

Let $\mathcal{G}$ be the union of the nonlinear components of $\mathcal{C}$. Our objective here is to show that $\mathcal{G}$ consists of only one absolutely irreducible nonlinear component.

In cases where $\{\eta(e_0),\eta(e_1),\eta(e_2)\}$ is either $\{0,0,0\}$ or $\{0,0,1\}$, we have that $\mathcal{G}=\mathcal{C}$ is a Fermat curve, for which  the absolute irreducibility  is well known.

The study of the remaining  cases is centered around a known result, namely  [\citealp{CV}, Theorem 1], which is presented in the appendix of this work in a slightly altered form as Theorem \ref{Thabs2}. This result is important because it  shows  that $\mathcal{G}$ is absolutely irreducible if
$$
\# \mathcal{G} (\mathbb{F}_q) \geqslant \frac{n(n+ q - 1)}{2} - \max\{n-1,2n-5\},
$$
where $n$ is the degree of $\mathcal{G}$. This is the case if and only if at least two of $ e_0,e_1$ and $e_2$ are nonzero, which includes all cases where 
$$
\# \mathcal{G} (\mathbb{F}_q) = \frac{1}{2} n(n + q - 1) - \frac{1}{2}(n-2).
$$

The following result uses notation given in Definition \ref{FrobDef}.
\begin{theorem} \label{Thcom2}
Each nonlinear absolutely irreducible component $\mathcal{X}\subseteq \mathcal{C}$ defined over $\F_q$ is an $\F_q$-Frobenius classical curve.
\end{theorem}
\begin{proof}
Let $Q(X_0,X_1,X_2) \in \mathbb{F}_q [X_0,X_1,X_2]$ be the absolutely irreducible polynomial such that $\mathcal{X}$ is given by $Q(X_0,X_1,X_2)=0$. It suffices to prove that $Q\nmid \Phi_q(Q)$.

Let  $R(X_0,X_1,X_2) \in \mathbb{F}_q [X_0,X_1,X_2]$ be such that $C=QR$, and note that 
$$
\Phi_q(C)=Q\Phi_q(R)+\Phi_q(Q)R.
$$ 
Therefore, it suffices to prove that $Q\nmid \Phi_q(C)$.

Since $C(X_0,X_1,X_2)=X_0^d+X_1^d+X_2^d+X_3^d$, with $X_3=e_0X_0+e_1X_1+e_2X_2$ and $d=(q-1)/2$, it follows that 
$$
\Phi_q(C)=d(X_0^{3d} + X_1^{3d} + X_2^{3d} + X_3^{3d}).
$$

From the polynomial identity 
\begin{eqnarray*}
X_0^{3d} + X_1^{3d} + X_2^{3d} + X_3^{3d}&=&C^3 - 3(X_0^d + X_1^d + X_2^d)X_3^d C-\\
& &-3(X_0^d + X_1^d)(X_0^d + X_2^d)(X_1^d + X_2^d),
\end{eqnarray*}
it follows that if $Q\mid \Phi_q(C)$, then $Q\mid (X_0^d + X_1^d)(X_0^d + X_2^d)(X_1^d + X_2^d)$, which contradicts $\deg Q>1$. 
Hence, $\mathcal{X}$ is $\F_q$-Frobenius classical.
\end{proof}

\subsection{Proof of Theorem \ref{Thmain}}\label{proofmainresult}

\begin{proof}
Given a specific set $\{ \eta(e_0),\eta(e_1),\eta(e_2) \}$, the number of $\mathbb{F}_q$-points with zero coordinates on $\mathcal{C}$, denoted by $M$, is listed in Table \ref{Tab-pts}, and the number $N_{(1)} + N_{(2)} + N_{(3)}$ of $\mathbb{F}_q$-points without zero coordinates on $\mathcal{C}$ is determined in Subsection \ref{App-pts-no} and summarized in Table \ref{tab:N(1)+N(2)+N(3)}. Hence,
$$
\# \mathcal{C} (\mathbb{F}_q) = N_{(1)} + N_{(2)} + N_{(3)} + M.
$$
By Subsection \ref{App-Lines}, the curve $\mathcal{C}$ has:
\begin{itemize}
\item exactly one linear component, if $d$ is odd and $\{ \eta(e_0),\eta(e_1),\eta(e_2) \}$ $=$ $\{ -1,1,1 \}$, in which case $\# \mathcal{G} (\mathbb{F}_q)$ $=$ $\# \mathcal{C} (\mathbb{F}_q) - (q+1)$ and $n = (q-3)/2$;
\item exactly two linear components, if $d$ is even and $\{ \eta(e_0),\eta(e_1),\eta(e_2) \}$ $=$ $\{ -1,-1,1 \}$, in which case $\# \mathcal{G} (\mathbb{F}_q)$ $=$ $\# \mathcal{C} (\mathbb{F}_q) - (2q+1)$ and $n = (q-5)/2$;
\item exactly three linear components, if $d$ is odd and $\{ \eta(e_0),\eta(e_1),\eta(e_2) \} = \{ -1,-1,-1 \}$, in which case $\# \mathcal{G} (\mathbb{F}_q) = \# \mathcal{C} (\mathbb{F}_q) - 3q$ and $n = (q-7)/2$.
\end{itemize}

In all other cases, $\mathcal{C}$ does not have a linear component, $\# \mathcal{G} (\mathbb{F}_q) = \# \mathcal{C} (\mathbb{F}_q)$, and $n = (q-1)/2$.

Tables \ref{Tab-odd} and \ref{Tab-even} summarize the number of linear components of $\mathcal{C}$, the degree $n$ of $\mathcal{G}$, and the value of $\#\mathcal{G}(\mathbb{F}_q)$ for the cases in which $d$ is odd and even. Note that rows $(1)$ and $(4)$ in Table \ref{Tab-odd} and rows $(2)$ and $(5)$ in Table \ref{Tab-even} present two classes of curves. A simple check shows that two curves arising from any of the two classes in a particular row are projectively equivalent.

\begin{table}[H]
\centering
\scalebox{0.9}{
\begin{tabular}{| c | c | c | c |c |} 
\hline
$     $	& $\{ \eta(e_0), \eta(e_1), \eta(e_2) \}$	& N		 & $n$		   & $\# \mathcal{G} (\mathbb{F}_q)$ \\ \hline
$(1)$	&$\{ 1,1,1 \}, \{ -1,-1,1 \}$               		& $0$             & $(q-1)/2$     & $n(n+q-1)/2 - 3(n-2)/2$ \\ \hline
$(2)$	&$\{ -1,1,1 \} $              				& $1$             & $(q-3)/2$     & $n(n+q-1)/2$ \\ \hline
$(3)$	&$\{ -1,-1,-1 \} $            				& $3$             & $(q-7)/2$     & $n(n+q-1)/2$ \\ \hline
$(4)$	&$\{ 0,1,1 \},\{-1,0,1 \}$               		& $0$             & $(q-1)/2$     & $n(n+q-1)/2 - (n-2)/2$ \\ \hline
$(5)$	&$\{ -1,-1,0 \} $             				& $0$             & $(q-1)/2$     & $n(n+q-1)/2 - 3(n-2)/2$ \\ \hline
$(6)$	&$\{ 0,0,1 \} $               				& $0$             & $(q-1)/2$     & $n(n+q-1)/2 - n(n-2)/2$ \\ \hline
$(7)$	&$\{ -1,0,0 \} $              				& $d$             & -                   & - \\ \hline
$(8)$	&$\{ 0,0,0 \} $               				& $0$             & $(q-1)/2$     & $n(n+q-1)/2 - 3n(n-2)/2$ \\ \hline
\end{tabular}
}
\caption{\small{Curve $\mathcal{G}$, for $d$ odd.}} 
\label{Tab-odd}
\end{table}

\vspace{-0.75cm}

\begin{table}[H]
\centering
\scalebox{0.85}{
\begin{tabular}{| c | c | c | c |c |} 
\hline
$ $&$\{ \eta(e_0), \eta(e_1), \eta(e_2) \}$     & N                & $n$               & $\# \mathcal{G} (\mathbb{F}_q)$ \\ \hline
$(1)$&$\{ 1,1,1 \} $               			& $0$             & $(q-1)/2$     & $n(n+q-1)/2$ \\ \hline
$(2)$&$\{ -1,1,1 \} ,\{ -1,-1,-1 \}$	           & $0$             & $(q-1)/2$     & $n(n+q-1)/2 - 3(n-2)/2$ \\ \hline
$(3)$&$\{ -1,-1,1 \} $             			& $2$             & $(q-5)/2$     & $n(n+q-1)/2$ \\ \hline
$(4)$&$\{ 0,1,1 \} $               			& $0$             & $(q-1)/2$     & $n(n+q-1)/2$ \\ \hline
$(5)$&$\{-1, 0,1 \},\{-1,-1, 0\} $ 	           & $0$             & $(q-1)/2$     & $n(n+q-1)/2 - (n-2)$ \\ \hline
$(6)$&$\{ 0,0,1 \} $               			& $0$             & $(q-1)/2$     & $n(n+q-1)/2 - n(n-2)/2$ \\ \hline
$(7)$&$\{ -1,0,0 \} $              			& $d$             & -                   & - \\ \hline
$(8)$&$\{ 0,0,0 \} $               			& $0$             & $(q-1)/2$     & $n(n+q-1)/2 - 3n(n-2)/2$ \\ \hline
\end{tabular}
}
\caption{\small{Curve $\mathcal{G}$, for $d$ even.}} 
\label{Tab-even}
\end{table}

\vspace{-0.75cm}

Statement (ii) follows directly from the results in Tables \ref{Tab-odd} and \ref{Tab-even}. 

Now the nonsingularity of $\mathcal{G}$ is discussed. The idea is to show that $\mathcal{G}$ attains the St\"ohr-Voloch bound (\ref{chapa}) and therefore show that it is nonsingular by Theorem \ref{nonsingularity}. 

If 
$$
\# \mathcal{G} (\mathbb{F}_q) = \frac{1}{2}n(n+q-1),
$$
then the proof is complete. 

In all other cases, $\mathcal{G}=\mathcal{C}$ and a direct calculation using the information in Tables \ref{tangentlines-odd} and \ref{tangentlines-even} shows that, for an $\mathbb{F}_q$-point $P$ with zero coordinates on $\mathcal{C}$ and tangent line $\ell_P$, the intersection multiplicity of $\mathcal{C}$ and $\ell_P$ at $P$ is $I(P, \mathcal{C} \cap \ell_P) = n$.
It follows immediately from Theorem \ref{nonsingularity} and Tables \ref{Tab-odd} and \ref{Tab-even} that the $\mathbb{F}_q$-points with zero coordinates on curve $\mathcal{C}$ are exactly its (total) $\mathbb{F}_q$-inflection points and that the St\"ohr-Voloch bound (\ref{chapa}) is attained.
Hence, $\mathcal{G}$ is nonsingular.

Finally, the classicality of the curve $\mathcal{G}$ is an immediate consequence of Theorem \ref{classicality}, since 
$$
\#\mathcal{G}(\F_q)>\frac{n(n+q-1)}{p^k},
$$
for every $k\geqslant 1$, and, by Theorem \ref{Thcom2}, the absolutely irreducible components defined over $\mathbb{F}_q$ of curve $\mathcal{G}$ satisfies the Frobenius classicality condition in Theorem \ref{Thabs2}. Further, the results in Tables \ref{Tab-odd} and \ref{Tab-even} show that $\mathcal{G}$ meets the conditions to be absolutely irreducible given in Theorem \ref{Thabs2} in cases where $\mathcal{G}$ is not a Fermat curve. Since the absolute irreducibility of Fermat curves  is well known, the proof of statement (i) is also complete.
\end{proof}

\begin{table}[H]
\centering
\scalebox{0.9}{
\begin{tabular}{|c|c|c|}
\hline
$\{\eta(e_0),\eta(e_1),\eta(e_2)\}$ & Points & Tangent lines \\ \hline
\multirow{3}{*}{$\eta(e_0)=1,\eta(e_1)=1,\eta(e_2)=1$}		& $P_{01}$					& $e_1^{d-1}X_0+e_0^{d-1}X_1=0$ \\
 										& $P_{02}$					& $e_2^{d-1}X_0+e_0^{d-1}X_2=0$ \\
 										& $P_{12}$					& $e_2^{d-1}X_1+e_1^{d-1}X_2=0$ \\\hline
\multirow{3}{*}{$\eta(e_i)=-1,\eta(e_j)=-1,\eta(e_k)=1$} 	& $P_i$			 		& $e_je_i^{d-1}X_j+e_ke_i^{d-1}X_k=0$ \\
 										& $P_j$ 					& $e_ie_j^{d-1}X_i+e_ke_j^{d-1}X_k=0$ \\
 										& $P_{ij}$					& $e_j^{d-1}X_i+e_i^{d-1}X_j$ \\ \hline
\multirow{1}{*}{$\eta(e_i)=0,\eta(e_j)=1,\eta(e_k)=1$}		& $P_{jk}$					& $e_k^{d-1}X_j+e_j^{d-1}X_k=0$ \\ \hline
\multirow{1}{*}{$\eta(e_i)=-1,\eta(e_j)=0,\eta(e_k)=1$}		& $P_i$					& $X_k=0$ \\ \hline
\multirow{3}{*}{$\eta(e_i)=-1,\eta(e_j)=-1,\eta(e_k)=0$}		& $P_i$					& $X_j=0$ \\
 										& $P_j$					& $X_i=0$ \\
 										& $P_{ij}$					& $e_j^{d-1}X_i+e_i^{d-1}X_j=0$ \\ \hline
\multirow{1}{*}{$\eta(e_i)=0,\eta(e_j)=0,\eta(e_k)=1$}		& Set $A_k$					& $\{x_i^{d-1}X_i+X_j=0;\eta(x_i)=-1\}$ \\ \hline
\multirow{3}{*}{$\eta(e_0)=0,\eta(e_1)=0,\eta(e_2)=0$}		& Set $A_0$					& $\{x_1^{d-1}X_1+X_2=0;\eta(x_1)=-1\}$ \\ 
										& Set $A_1$					& $\{x_0^{d-1}X_0+X_2=0;\eta(x_0)=-1\}$ \\
										& Set $A_2$					& $\{x_0^{d-1}X_0+X_1=0;\eta(x_0)=-1\}$\\\hline
\end{tabular}
}
\caption{\small{$\mathbb{F}_q$-points $P$ with zero coordinates on curve $\mathcal{C}$ with their respective tangent lines, for $d$ odd and $\{\eta(e_i),\eta(e_j),\eta(e_k)\}=\{\eta(e_0),\eta(e_1),\eta(e_2)\}$, only in cases where $\# \mathcal{G} (\mathbb{F}_q) < \frac{1}{2} n(n+q-1)$.}}
\label{tangentlines-odd}
\end{table}

\begin{table}[H]
\centering
\scalebox{0.9}{
\begin{tabular}{|c|c|c|}
\hline
$\{\eta(e_0),\eta(e_1),\eta(e_2)\}$ & Points & Tangent lines \\ \hline
\multirow{3}{*}{$\eta(e_i)=-1,\eta(e_j)=1,\eta(e_k)=1$}		& $P_i$					& $e_je_i^{d-1}X_j+e_ke_i^{d-1}X_k=0$ \\
 										& $P_{ij}$ 					& $e_j^{d-1}X_i-e_i^{d-1}X_j=0$ \\
 										& $P_{ik}$ 					& $e_k^{d-1}X_i-e_i^{d-1}X_k=0$ \\\hline
\multirow{3}{*}{$\eta(e_0)=-1,\eta(e_1)=-1,\eta(e_2)=-1$} 	&$P_0$					& $e_1e_0^{d-1}X_1+e_2e_0^{d-1}X_2=0$ \\
 										& $P_1$ 					& $e_0e_1^{d-1}X_0+e_2e_1^{d-1}X_2=0$ \\
 										& $P_2$ 					& $e_0e_2^{d-1}X_0+e_1e_2^{d-1}X_1=0$ \\ \hline
\multirow{2}{*}{$\eta(e_i)=-1,\eta(e_j)=0,\eta(e_k)=1$} 		& $P_i$ 					& $X_k=0$ \\ 
 										& $P_{ik}$ 					& $e_k^{d-1}X_i-e_i^{d-1}X_k=0$ \\ \hline
\multirow{3}{*}{$\eta(e_i)=-1,\eta(e_j)=-1,\eta(e_k)=0$} 	& $P_i$ 					& $X_j=0$ \\
 										& $P_j$ 					& $X_i=0$ \\ \hline
\multirow{1}{*}{$\eta(e_i)=0,\eta(e_j)=0,\eta(e_k)=1$}		& Set $A_k$					& $\{x_i^{d-1}X_i+X_j=0;\eta(x_i)=-1\}$ \\ \hline
\multirow{3}{*}{$\eta(e_0)=0,\eta(e_1)=0,\eta(e_2)=0$}		& Set $A_0$					& $\{x_1^{d-1}X_1+X_2=0;\eta(x_1)=-1\}$ \\ 
										& Set $A_1$					& $\{x_0^{d-1}X_0+X_2=0;\eta(x_0)=-1\}$ \\
										& Set $A_2$					& $\{x_0^{d-1}X_0+X_1=0;\eta(x_0)=-1\}$\\\hline
\end{tabular}
}
\caption{\small{$\mathbb{F}_q$-points $P$ with zero coordinates on curve $\mathcal{C}$ with their respective tangent lines, for $d$ even and $\{\eta(e_i),\eta(e_j),\eta(e_k)\}=\{\eta(e_0),\eta(e_1),\eta(e_2)\}$, only in cases where $\# \mathcal{G} (\mathbb{F}_q) < \frac{1}{2} n(n+q-1)$.}}
\label{tangentlines-even}
\end{table}

\vspace{-1.25cm}

\appendix

\section{Rudiments of the St\"ohr-Voloch Theory}

In this appendix, we present  some notions and basic facts from the St\"ohr-Voloch Theory.
We believe the results here are well known by the specialists. Nevertheless, as some of them are not explicitly stated in the literature, we provide their proofs.

\begin{definition}\label{FrobDef} 
For any homogeneous polynomial $Q(X_0,X_1,X_2) \in \F_{q}[X_0,X_1,X_2]$, consider the polynomial 
\begin{equation}\label{FrobCondition}
\Phi_q(Q):=X_0^{q}Q_{X_0}+X_1^{q}Q_{X_1}+X_2^{q}Q_{X_2},
\end{equation}
where $Q_{X_0}, Q_{X_1}$ and $Q_{X_2}$ are the formal partial derivatives of $Q(X_0,X_1,X_2)$. An absolutely irreducible curve $\mathcal{F} : F(X_0,X_1,X_2) = 0$ defined over $\F_q$ is called $\F_q$-Frobenius nonclassical if 
\begin{equation}\label{Frob}
F \text{  divides }  \Phi_q(F).
\end{equation}
Otherwise, $\mathcal{F}$  is called  $\F_q$-Frobenius classical.
\end{definition}

Note that \eqref{Frob} has a geometric meaning; that is, the Frobenius map $P\mapsto P^q$ takes each simple point $P \in \mathcal{F}$ to the tangent line to $\mathcal{F}$ at $P$.

\begin{theorem} \label{Thabs2}
Let $\mathcal{F}$ be a plane algebraic curve of degree $m$ defined over $\mathbb{F}_q$. If all absolutely irreducible components of $\mathcal{F}$ defined over $\mathbb{F}_q$ are $\F_q$-Frobenius classical, then $$\# \mathcal{F} (\mathbb{F}_q) \leqslant \frac{m(m+ q -1)}{2}.$$ Moreover, if  $\# \mathcal{F} (\mathbb{F}_q) \geqslant \frac{m(m + q - 1)}{2} - \max \{m-1,2m-5\},$ then $\mathcal{F}$ is absolutely irreducible.
\end{theorem}

\begin{proof}
The first bound is as presented in [\citealp{CV}, Theorem 1], but the second improvement is made by changing line 14 of that proof from
$$
\sum_{i<j} m_i m_j \geqslant m_1(m_2 + \cdots + m_s) \geqslant 2(m_2 + \cdots + m_s) \geqslant m_1 + \cdots +m_s \geqslant m
$$
to
$$
\sum_{i<j} m_i m_j \geqslant  m_1(m_1 + \cdots + m_s) - m_1 m_1=m_1m - m_1^2 \geqslant 2m - 4,
$$
where the last inequality follows from the fact that $m_1^2-m_1m+2m-4\leqslant 0$ for $2\leqslant m_1\leqslant m-2$.
\end{proof}

The proofs of the nonsingularity and the classicality of curve $\mathcal{G}$ are based on Theorems \ref{nonsingularity} and \ref{classicality}, which are slightly adapted versions of [\citealp{SV}, Theorem 0.1] and [\citealp{HK}, Theorem 1.3], respectively.

In what follows, $I(P, \mathcal{X} \cap \mathcal{Y})$ denotes the intersection multiplicity of curves $\mathcal{X}$ and $\mathcal{Y}$ at a point $P$.

\begin{theorem}\label{nonsingularity}
Let $\mathcal{F}$ be an $\F_q$-Frobenius classical plane curve of degree $m$ and let $\mathfrak{P}_1,\mathfrak{P}_2,\ldots, \mathfrak{P}_s$ be the distinct inflection points of $\mathcal{F}$ defined over $\F_{q}$. If $\mathcal{L}_i$ is the tangent line to $\mathcal{F}$ at $\mathfrak{P}_i$ and $\mathfrak{m}_i=I(\mathfrak{P}_i, \mathcal{F} \cap \mathcal{L}_i)\geqslant 3$ is the intersection multiplicity of $\mathcal{F}$ and $\mathcal{L}_i$ at $\mathfrak{P}_i$, for $i=1,\ldots,s$, then
\begin{equation}\label{chapa}
\#\mathcal{F}(\F_q)\leqslant \frac{m(m+q-1)-\sum_{i=1}^{s}(\mathfrak{m}_i-2)}{2}.
\end{equation}
If equality holds in \eqref{chapa}, then $\mathcal{F}$ is nonsingular.
\end{theorem}

\begin{proof}
Let  $\mathcal{H}$ be the curve given by $\Phi_q(F)=0$. Since $\mathcal{F}$ is $\mathbb{F}_q$-Frobenius classical, B\'{e}zout's Theorem gives $$\sum_{P \in \mathcal{F}\cap\mathcal{H}}I(P, \mathcal{F}\cap\mathcal{H})=m(m+q-1).$$ 

From Euler's Formula, $\mathcal{F}(\mathbb{F}_q)\subseteq \mathcal{H}(\mathbb{F}_q)$ and therefore 
\begin{eqnarray}
\label{princexpr}
\sum_{P \in \mathcal{F}(\mathbb{F}_q)}I(P, \mathcal{F}\cap\mathcal{H})\leqslant \sum_{P \in \mathcal{F}\cap\mathcal{H}}I(P, \mathcal{F}\cap\mathcal{H})=m(m+q-1).
\end{eqnarray} 

If $P \in\mathcal{F}(\mathbb{F}_q)$ is nonsingular and $\mathcal{L}$ is the tangent line to $\mathcal{F}$ at $P$, then direct computation shows that  $I(P, \mathcal{H} \cap \mathcal{L}) \geqslant I(P, \mathcal{F} \cap \mathcal{L})$. In particular, from [\citealp{FCHN}, Lemma 3.3],

\[ 
I(P,\mathcal{F}\cap\mathcal{H})\geqslant  
\begin{cases}
 2, & \text{ for all } P\in \mathcal{F}(\mathbb{F}_q),\\
\mathfrak{m}_i, &  \text{ if } P=\mathfrak{P}_i.
\end{cases}
\]

Therefore,

\begin{eqnarray*}
\sum_{P \in \mathcal{F}(\mathbb{F}_q)}I(P, \mathcal{F}\cap\mathcal{H})&=&\sum_{i=1}^{s}I(\mathfrak{P}_i, \mathcal{F}\cap\mathcal{H})+\sum_{P \in \mathcal{F}(\mathbb{F}_q)\setminus \{\mathfrak{P}_1,\ldots,\mathfrak{P}_s\}}I(P, \mathcal{F}\cap\mathcal{H})\\
&\geqslant&\sum_{i=1}^{s}(\mathfrak{m}_i-2)+2\cdot\#\mathcal{F}(\mathbb{F}_q),
\end{eqnarray*}
and thus \eqref{princexpr} gives $$\#\mathcal{F}(\mathbb{F}_q)\leqslant \frac{m(m+q-1)-\sum_{i=1}^{s}(\mathfrak{m}_i-2)}{2}.$$

Assuming equality in \eqref{chapa}, we have that $$\sum_{i=1}^{s}\mathfrak{m}_i+2\cdot\#\mathcal{F}(\mathbb{F}_q)\setminus\{\mathfrak{P}_1,\ldots,\mathfrak{P}_s\}=m(m+q-1)=\sum_{P \in \mathcal{F}\cap\mathcal{H}}I(P, \mathcal{F}\cap\mathcal{H}),$$
and the following holds:
\begin{enumerate}[\rm(i)]
\item $I(P, \mathcal{F}\cap\mathcal{H})=2$, for all $P \in \mathcal{F}(\mathbb{F}_q)\setminus\{\mathfrak{P}_1,\ldots,\mathfrak{P}_s\}$;
\item $I(\mathfrak{P}_i, \mathcal{F}\cap\mathcal{H})=\mathfrak{m}_i$, for all $i=1,\ldots,s$;
\item $\mathcal{F}\cap\mathcal{H}=\mathcal{F}(\mathbb{F}_q)$.
\end{enumerate}

From the definition of $\mathcal{H}$, it follows that any singular point $P\in  \mathcal{F}$  must be a point of  $\mathcal{H}$ and then assertion (iii) gives that $P \in\mathcal{F}(\mathbb{F}_q)$.
However, if $P \in\mathcal{F}(\mathbb{F}_q)$  is singular, then $P \in\mathcal{H}(\mathbb{F}_q)$  is singular and then $I(P,\mathcal{F}\cap\mathcal{H})\geqslant 2\cdot 2=4$, which contradicts assertion (i). 
Hence, $\mathcal{F}$ is a nonsingular curve.
\end{proof}

\begin{theorem}\label{classicality}
Let $\mathcal{F}: F(X_0,X_1,X_2)=0$ be an $\F_q$-Frobenius classical curve of degree $m<q$. If $\mathcal{F}$ has an infinite number of inflection points, then 
\begin{eqnarray}
\label{boundnonclassicality}
\#\mathcal{F}(\F_q)\leqslant \frac{m(m+q-1)}{p^k}
\end{eqnarray}
for some $k\geqslant 1$.
\end{theorem}
\begin{proof}
Considering the notation as in \cite{HK}, let $B_q$ be the number of branches of $\mathcal{F}$ centered at points of $\mathbb{P}^2(\mathbb{F}_q)$. Therefore, since $\#\mathcal{F}(\F_q)\leqslant B_q$, all that is needed is to show that for some $k\geqslant 1$ 

\begin{eqnarray}
\label{upper}
B_q\leqslant \frac{(2g-2)+(q+2)m}{p^k}\leqslant \frac{m(m+q-1)}{p^k},
\end{eqnarray}
where $g$ is the genus of $\mathcal{F}$. Note that the second inequality in \eqref{upper} is trivial. The first one is proved by considering $\nu=1$ in inequality (3.1) in the proof of [\citealp{HK}, Theorem 1.3]. In fact, since $\mathcal{F}$ is $\F_q$-Frobenius classical, 
$$8
v_P(S)\geqslant v_P([\overline{x}(t)D_t^{(1)}\overline{y}(t)-\overline{y}(t)D_t^{(1)}\overline{x}(t)])\geqslant r+s-1\geqslant \varepsilon,
$$  
if $v_P(S)<rq$, and 
$$
v_P(S)\geqslant rq > rm \geqslant rs \geqslant \varepsilon
$$
otherwise. Hence, as $\mathcal{F}$ has an infinite number of inflection points, and then $\varepsilon = p^k$ for some $k\geqslant 1$ (see [\citealp{GV}, Proposition 2]), the first inequality in \eqref{upper} follows.
\end{proof}

\begin{acknowledgements*}
The first author was supported by FAPESP (Brazil), grant 2017/04681-3. The third author was supported by CNPq (Brazil), grant 154359/2016-5.
\end{acknowledgements*}


\end{document}